\documentclass[a4paper,10pt,reqno]{amsart}

\usepackage{amsmath, amstext, paralist, amsthm, amssymb, epsfig, paralist, tikz, tikz-cd, graphicx}
\usepackage{hyperref}
\usepackage{xurl}

\newtheorem{thm}{Theorem}[section]
\newtheorem{lem}[thm]{Lemma}
\newtheorem{prop}[thm]{Proposition}
\newtheorem{ex}[thm]{Example}

\renewcommand{\bar}{\overline}

\newcommand{\Ker}{\operatorname{Ker}\nolimits}
\newcommand{\Coker}{\operatorname{Cok}\nolimits}
\newcommand{\Cok}{\operatorname{Cok}\nolimits}
\newcommand{\cok}{\operatorname{cok}\nolimits}
\newcommand{\coker}{\operatorname{cok}\nolimits}

\renewcommand{\Im}{\operatorname{Im}\nolimits}
\newcommand{\im}{\operatorname{im}\nolimits}
\newcommand{\coim}{\operatorname{coim}\nolimits}
\newcommand{\Coim}{\operatorname{Coim}\nolimits}

\begin{document}
	\title{Some Remarks about Integral Categories}
	\author{Yaroslav Kopylov}
	\address{Yaroslav Kopylov\,
		%\orcidlink{0000-0002-0343-4424}\,
		%0000-0002-0343-4424
		\newline\hphantom{iii} Sobolev Institute of Mathematics
		\newline\hphantom{iii} 4 Koptyug Ave.
		\newline\hphantom{iii} 630090, Novosibirsk, Russia}
	\email{\href{mailto:yakop@math.nsc.ru}{yakop@math.nsc.ru}}
	
	\author{Max Zinchenko}
	\address{Max Zinchenko\,
		%\orcidlink{???}\,
		%???
		\newline\hphantom{iii} Novosibirsk State University
		\newline\hphantom{iii} 1 Pirogova St.
		\newline\hphantom{iii} 630090, Novosibirsk, Russia}
	\email{\href{mailto:m.zinchenko@g.nsu.ru}{m.zinchenko@g.nsu.ru}}
	
	\renewcommand{\thefootnote}{}
	\vspace{-35pt}
	
	{\small
		\begin{abstract}\noindent{}
			We consider integral categories, a class of categories that  gained importance in the last three decades in the categorical foundations of algebra and functional analysis. We discuss some familiar criteria for integrality and prove new criteria for one- and two-sided integrality. For the first time in the literature, an example is given of a right but not left integral category.
	\end{abstract}}
	
	\maketitle
	
	\section{Introduction}\label{Introduction}\vspace{-5pt}
	
	Integral categories were first considered by B\u anic\u a and Popescu in 1965 in \cite{BanicaPopescu1965} under the~name {\em pr\'eab\'eliennes sp\'eciales}. These are preabelian categories (i.e., additive categories with kernels and cokernels) in which epimorphisms are stable under pullbacks and monomorphisms are stable under pushouts. In 2001, Rump called such categories integral and introduced the notions of a left integral category (a~preabelian one in which epimorphisms are stable under pullbacks) and a~right integral category (a~preabelian one in which monomorphisms are stable under pushouts).  
	
	Integral categories have recently been considered in representation theory (see, for example, \cite{BrHasTat2021, BuanMarsh2012, Nakaoka2013} and in functional analysis (see
	\cite{HSW2021, LawsonWegner2023}).
	
	Numerous examples and non-examples of integral categories can be found in \cite{HSW2021}. However, no example seems to have been known of a right integral category that is not left integral (or vice versa).  
	
	The structure of the paper is as follows. In Section~\ref{Prep}, 
	we give the necessary definitions and discuss the interrelationship between the types of preabelian categories we consider. In Section \ref{CrInt}, we discuss some familiar criteria for integrality and prove some extra criteria for one- and two-sided integrality of a~preabelian category. Then, in Section~\ref{ex-couple}, we adapt the results of Kopylov and Wegner \cite{KW2} on the derivation of exact couples in a semi-abelian category in the sense of Raikov (using the approach of Eckmann, Hilton \cite{EH} for constructing exact couples, see~\cite{M}), to integral categories. In the~last section, Section~\ref{Example}, we give an~explicit example of a~right integral category that is not left integral and discuss its consequences.
	
	\section{Notations and Preliminaries}\label{Prep}\vspace{-5pt}
	
	In the sequel, let $\mathcal{A}$ be a preabelian category, i.e.~an additive category with kernels and cokernels. For a morphism $f\colon E\rightarrow F$ in $\mathcal{A}$ we denote by $\ker f\colon \Ker f\rightarrow E$ its kernel and by $\cok f\colon F \rightarrow \Cok f$ its cokernel. Note that kernels and cokernels are unique only up to isomorphisms; we will however speak of \textit{the} kernel and \textit{the} cokernel of a given morphism $f$. We say that $f$ is a \textit{kernel} if there is a morphism $g$ such that $f=\ker g$. \textit{Cokernels} are defined dually. We denote by $\coim f\colon E\rightarrow \Coim f$ the cokernel of $\ker f$ and by $\im f\colon\Im f\rightarrow F$ the kernel of $\cok f$. As above, \textit{image} and \textit{coimage} are unique only up to isomorphims but we will also use definite articles in the sequel; the~same refers to pullbacks and pushouts. Any morphism $f$ in a~preabelian category admits a canonical decomposition $f=(\im f)\bar{f}\coim f$. Following B\u{a}nic\u{a}, Popescu \cite{BanicaPopescu1965} or Schneiders \cite[Definition 1.1.1]{Schneiders1999}, we say that $f$ is \textit{strict} if $\bar{f}$ is an isomorphism. By definition, $\mathcal{A}$ is abelian if and only if every morphism is strict.
	We say that a kernel $\alpha$ is semi-stable if all its pushouts along arbitrary morphisms are again kernels. Semi-stable cokernels are defined dually.
	
	\smallskip
	
	According to Rump \cite[p.~167]{Rump2001}, we say that $\mathcal{A}$ is \textit{left semi-abelian} if $\bar{f}$ is a monomorphism for each morphism $f$. Dually, we say that $\mathcal{A}$ is \textit{right semi-abelian} if $\bar{f}$ is an epimorphism for each morphism $f$. If $\mathcal{A}$ is left semi-abelian then each morphism $f$ admits a decomposition $f=ip$ with a cokernel $p$ and a monomorphism $i$. If dually $\mathcal{A}$ is right semi-abelian then each morphism $f$ admits a decomposition $f=ip$ with an epimorphism $p$ and a kernel $i$. In fact, the last assertions are even equivalent to the definitions of left and right semi-abelianity (see \cite[p.~167]{Rump2001}). A~left and right semi-abelian category
	is called a~semi-abelian category (in the~sense of Palamodov):~$\mathcal{A}$ is semi-abelian iff $\bar{f}$ is a \textit{bimorphism} (=\textit{regular morphism}), i.e.~$\bar{f}$ is a monomorphism and an epimorphism simultaneously, for each morphism $f$. 
	
	We need some results on preabelian categories. The first lemma summarizes well-known facts, see, e.g., \cite[Remark 1.1.2]{Schneiders1999}, \cite[Theorems 1 and 5]{RichmanWalker}, and \cite[Proposition 5.2]{Kelly1969}.
	
	\begin{lem}\label{l1} Let $\mathcal{A}$ be a preabelian category.
		\begin{itemize}
			\item[(i)] A morphism $f$ is a kernel if and only if $f=\im f$ and it is a cokernel if and only if $f=\coim f$.
			\item[(ii)] A morphism $f$ is strict if and only if there is a representation $f=f_1 f_0$ with a cokernel $f_0$ and a kernel $f_1$. In every such representation we have
			$f_0=\coim f$ and $f_1=\im f$.
			\item[(iii)]In every pullback
			$$
			\begin{tikzcd}
				\Ker p_G \arrow[r] & P \arrow[r, "p_G"] \arrow[d, "p_E"'] \arrow[r, "\text{PB}", phantom, bend right=60] & G \arrow[d, "t"]  &   \\
				\Ker f \arrow[r]   & E \arrow[r, "f"']                                                                   & F \arrow[r, "h"'] & H
			\end{tikzcd}
			$$
			we have $\ker f= p_E \ker p_G$. If $f$ is the kernel of a morphism $h$ then $p_G$ is the kernel of $ht$. If $f$ is a monomorphism then so is $p_G$.
			
			\smallskip
			
			Dually, in every pushout
			$$
			\begin{tikzcd}
				& E \arrow[r, "g"] \arrow[d, "s"'] & F \arrow[d, "s_F"] \\
				& G \arrow[r, "s_G"'] & S
				\arrow[from=1-2, to=2-3, phantom, "\text{\small PO}" description]
			\end{tikzcd}
			$$
			we have $\cok g=(\cok s_G) s_F$. If $g$ is the cokernel of a morphism $h$ then $s_G$ is the cokernel of $sh$. If $g$ is an epimorphism then so is $s_G$.\hfill\qed
		\end{itemize}
	\end{lem}
	
	We stress that the last statement of Lemma \ref{l1} implies that in any preabelian category, kernels pullback to kernels and cokernels pushout to cokernels.
	
	Following Rump \cite{Rump2001}, we say that a~preabelian category
	$\mathcal{A}$ is called {\em left integral} if pullbacks of epimorphisms in~$\mathcal{A}$ are epimorphisms. Dually,
	$\mathcal{A}$ is {\em right integral} if pushouts of monomorphisms are monomorphisms. 
	
	It is clear that a~left (resp., right) integral category is left (resp., right) semi-abelian. 
	
	A~left and right integral category is called {\em integral}. 
	
	A {\it left quasi-abelian category} is a preabelian category in which pullbacks of cokernels are cokernels; a {\em right quasi-abelian category} is a preabelian category in which pushouts of kernels are kernels. A left and right quasi-abelian category is called {\em quasi-abelian}.

	\section{Criteria for Integrality}\label{CrInt}
	
	In this section, we discuss some familiar criteria for integrality and prove some extra criteria for one- and two-sided integrality. 
	
	\smallskip
	
	Given a commutative square
	\begin{equation}\label{sq1}
		\begin{tikzcd}
			\Ker g \arrow[r] \arrow[d, "\hat{\alpha}"'] & C \arrow[r, "g"] \arrow[d, "\alpha"'] & D \arrow[d, "\beta"] \arrow[r] & \Cok g \arrow[d, "\hat{\beta}"] \\
			\Ker f \arrow[r]                            & A \arrow[r, "f"']                     & B \arrow[r]                    & \Cok f                         
		\end{tikzcd}
	\end{equation}
	in a preabelian category $\mathcal{A}$, we denote by $\hat\alpha\colon\Ker g\rightarrow\Ker f$ the unique morphism satisfying $\alpha\ker g=(\ker f)\hat\alpha$ and dually by $\hat\beta\colon\Coker g\rightarrow\Coker f$ the unique morphism satisfying $(\coker f)\beta=\hat\beta\coker g$. The above square will be used throughout the remainder of the section without repetition. 
	
	The following proposition gives us a~criterion for a~preabelian category to be integral, which will be used below:
	
	\begin{prop}\label{t1.1} 
		The following are equivalent for preabelian category $\mathcal{A}$:
		\begin{itemize}
			\item[(i)] $\mathcal{A}$ is integral.
			\item[(ii)] If ~\eqref{sq1} is a pushout then the canonical morphism $\hat\alpha: \Ker g \to \Ker f$ is an epimorphism, and if~\eqref{sq1}
			is a~pullback then the~canonical morphism $\hat\beta: \Coker g \to \Coker f$ is a~monomorphism.
		\end{itemize}
	\end{prop}
	
	\begin{proof}
		(i)$\Rightarrow$(ii) 
		Since $\mathcal A$ is integral, it is also semi-abelian. Thus, we can decompose $g$ as $g=g_1 g_0$ with $g_1$ a monomorphism and $g_0$ a cokernel. Consider the pushout
		$$
		\begin{tikzcd}
			C \arrow[r, "g_0"] \arrow[d, "\alpha"'] \arrow[r, "\text{PO}", phantom, bend right=60] & A' \arrow[d, "\beta'"] \arrow[r, "g_1", hook] & D \arrow[d, "\beta"] \\
			A \arrow[r, "f_0"'] \arrow[rr, "f"', bend right=49]                                    & B' \arrow[r, "f_1"', dashed]                  & B                   
		\end{tikzcd}
		$$
		We have $\beta g=\beta g_1 g_0=f\alpha$. Thus, there is a unique morphism $f_1\colon B'\rightarrow B$ with $f=f_1 f_0$ and $\beta g_1=f_1 \beta'$. It is easy to see (see, for instance, \cite[Lemma 5.1]{Kelly1969}) that
		$$
		\begin{tikzcd}
			A' \arrow[r, "g_1"] \arrow[d, "\beta' "'] & D \arrow[d, "\beta"] \\
			B' \arrow[r, "f_1"'] & B
			\arrow[from=1-1, to=2-2, phantom, "\scriptstyle\text{PO}" description]
		\end{tikzcd}
		$$
		is again a pushout and so $f_1$ is a monomorphism because $\mathcal{A}$ integral.\\
		Let $x\colon\Ker f\rightarrow X$ be a morphism such that $x\hat\alpha=0$. Consider the pushout
		$$
		\begin{tikzcd}
			\Ker g \arrow[r, "\hat{\alpha}"] \arrow[d] & \Ker f \arrow[r, "x"] \arrow[d, "\ker f"'] \arrow[r, "\text{PO}", phantom, bend right=60] & X \arrow[d, "z_2"] \\
			C \arrow[r, "\alpha"']                     & A \arrow[r, "z_1"']                                                                       & Z                 
		\end{tikzcd}
		$$
		We have $z_1\alpha\ker g=z_1(\ker f)\hat\alpha = z_2 x \hat\alpha=0$ and $g_0=\coker\ker g$. Thus, there exists a unique morphism $w\colon A'\rightarrow Z$ such that $z_1\alpha=w g_0$. Since $\beta' g_0=f_0\alpha$ is a pushout, there is a unique morphism $\sigma\colon B'\rightarrow Z$ such that $w=\sigma\beta'$ and $z_1=\sigma f_0$. But then $z_2 x=z_1 \ker f=\sigma f_0 \ker f=0$ because $0=f\ker f=f_1 f_0 \ker f$ and $f_1$ is a monomorphism. Since $\ker f$ is a monomorphism and the category is right integral, $z_2$ is a monomorphism, whence $x=0$ and $\hat\alpha$ is an epimorphism. 
		
		By duality, we see that $\hat\beta$ is a~monomorphism.
		\smallskip
		
		(ii)$\Rightarrow$(i) If (1) is a pushout with $g$ a monomorphism then $(\ker f)\hat\alpha = \alpha\ker g=0$. Since $\hat\alpha$ is an epimorphism, this gives $\ker f=0$, i.e.~$f$ is a monomorphism. Hence, $A$ is right integral. By duality, we conclude that $\mathcal{A}$ is also left integral. 
	\end{proof}
	
	The~following Proposition is well known (see~\cite[Proposition~7]{BanicaPopescu1965}, \cite[Proposition~6]{Rump2001}, and \cite[Proposition~4.1]{BuanMarsh2012}).
	
	\begin{prop}\label{sa-int1}
		Let $\mathcal{A}$ be a semi-abelian category. The~following are equivalent:
		\item[(i)] $\mathcal{A}$ is integral.
		\item[(ii)] Bimorphisms in $\mathcal{A}$ are stable under pullbacks.
		\item[(iii)] Bimorphisms in $\mathcal{A}$ are stable under pushouts.
		\item[(iv)] The~system of all bimorphisms in $\mathcal{A}$ admits
		a~calculus of left fractions and a~calculus of right fractions.
	\end{prop}
	
	In the~following proposition, we suggest a~criterion 
	for the~right (resp., left) integrality of a~right (resp., left) semi-abelian category. 
	
	\begin{prop}\label{sa-int2}
		A~right semi-abelian category $\mathcal{A}$ is right integral if and only if pushouts of monomorphisms along epimorphisms in~$\mathcal{A}$ are monomorphisms.
		
		Dually, a~left semi-abelian category $\mathcal{A}$ is left integral if and only if pullbacks of epimorphisms along monomorphisms in~$\mathcal{A}$ are epimorphisms.
	\end{prop}
	
	\begin{proof}
		We prove the~first assertion.
		
		If $\mathcal{A}$ is right integral then the~assertion is obvious.
		
		Suppose that pushouts of monomorphisms along epimorphisms in~$\mathcal{A}$ are monomorphisms and consider a~pushout
		$$
		\begin{tikzcd}
			C \arrow[r, "g"] \arrow[d, "\alpha"'] & D \arrow[d, "\beta"] \\
			A \arrow[r, "f"'] & B
			\arrow[from=1-1, to=2-2, phantom, "\scriptstyle\text{PO}" description]
		\end{tikzcd}
		$$
		with $g$ a~monomorphism. Since $\mathcal{A}$ is right semi-abelian,
		$\overline{\alpha}$ is an~epimorphism, and so $\alpha=\alpha_1 \alpha_0$, where $\alpha_1:A'\to A$ is a~kernel and $\alpha_0:C\to A'$ is an~epimorphism. Consider the~pushout
		$$
		\begin{tikzcd}
			C \arrow[r, "g"] \arrow[d, "\alpha_0"'] & D \arrow[d, "\beta_0"] \\
			A' \arrow[r, "f_0"'] & B'
			\arrow[from=1-1, to=2-2, phantom, "\scriptstyle\text{PO}" description]
		\end{tikzcd}
		$$
		Since $f\alpha_1\alpha_0 = f\alpha = \beta g$ and the~above square
		is a~pushout, there exists a~unique morphism $\beta_1:B'\to B$ such that $f\alpha_1 = \beta_1 f_0$ and $\beta=\beta_1\beta_0$. Since $\alpha_0$ is an~epimorphism, it is known and easy to check that
		$$
		\begin{tikzcd}
			A' \arrow[r, "f_0"] \arrow[d, "\alpha_1"'] & A \arrow[d, "\beta_1"] \\
			A \arrow[r, "f"'] & B
			\arrow[from=1-1, to=2-2, phantom, "\scriptstyle\text{PO}" description]
		\end{tikzcd}
		$$
		is also a~pushout. Since $\mathcal{A}$ is right semi-abelian and $\alpha_1$ is a~kernel, we conclude that the~last square is a~pullback (see \cite[Proposition~3.1]{KW}). Thus, $0=\Ker f_0\equiv \Ker f$, and $f$ is a~monomorphism. 
	\end{proof}

	\section{Exact couples in an~integral category}\label{ex-couple}
	
	An exact couple in $\mathcal{A}$ is a diagram of the form
	\begin{equation}\label{ex-coup}
		\begin{tikzcd}
			D \arrow[]{rr}{\alpha}& &D \arrow[]{ld}{\beta}\\
			&E \arrow[]{ul}{\gamma}
		\end{tikzcd}
	\end{equation}
	such that $\im\alpha=\ker\beta$, $\im\beta=\ker\gamma$ and $\im\gamma=\ker\alpha$. Using Lemma~1 in \cite{Yak1979}, it is not hard to see that the last equalities are equivalent to $\cok\alpha=\coim\beta$, $\cok\beta=\coim\gamma$ and $\cok\gamma=\coim\alpha$, respectively. For instance, $\coim\beta=\cok\im\alpha=\cok\alpha$ and the dual computation yield the first equivalence.
	
	\begin{thm}\label{THM-1}\cite[Theorem 1.]{KW2} Let $\mathcal{A}$ be semiabelian and consider the exact couple \eqref{ex-coup}. Assume that $\alpha$, $\beta$ and $\gamma$ are strict. Then the Eckmann-Hilton construction below gives rise to the following two diagrams
		\begin{equation}\label{der-coup}
			\begin{tikzcd}
				D_1 \arrow[]{rr}{\alpha_1}& &D_1 \arrow[]{ld}{\beta_1^-}\\
				&E_1^- \arrow[]{ul}{\gamma_1^-}
			\end{tikzcd}\hspace{20pt}
			\begin{tikzcd}
				D_1 \arrow[]{rr}{\alpha_1}& &D_1 \arrow[]{ld}{\beta_1^+}\\
				&E_1^+ \arrow[]{ul}{\gamma_1^+}
			\end{tikzcd}
		\end{equation}
		which we call the \emph{left} resp.~the \emph{right derived couple}. The diagrams in \eqref{der-coup} have the following properties.\vspace{3pt}
		
		\begin{itemize}
			\item[(i)] Both diagrams are exact couples.
			
			\item[(ii)] With $\partial=\beta\gamma$ we get $H^-(E,\partial)=E_1^-$ and $H^+(E,\partial)=E_1^+$. Here, the \emph{left} resp.~\emph{right cohomology} is defined via $H^-(E,\partial)=\Cok(\theta\colon\Im\partial\rightarrow\Ker\partial)$ resp.~$H^+(E,\partial)=\Ker(\tau\colon\Cok\partial\rightarrow\Coim\partial)$, where $\theta$ and $\tau$ are the natural maps.
			
			\item[(iii)] There is a canonical bimorphism $\omega\colon E_1^-\rightarrow E_1^+$ satisfying the equations $\omega\beta_1^-=\beta_1^+$, $\gamma_1^-=\gamma_1^+\omega$ and $(\ker\tau)\omega(\cok\theta)=(\cok\partial)(\ker\partial)$. The morphism $\omega$ is uniquely determined by the third equation and it is an isomorphism if $\ker\partial$ or $\cok\partial$ is semistable.
		\end{itemize}
	\end{thm}
	
	\begin{thm}\label{THM-2}\cite[Theorem 2.]{KW2} Let $\mathcal{A}$ be semiabelian and consider the exact couple \eqref{ex-coup}. Assume that $\beta$ and $\gamma$ are strict and that $\ker\gamma$ and $\cok\beta$ are semistable. If the powers $\alpha^k$ are strict for $1\leqslant{}k\leqslant{}n$, then the derivation process of Theorem \ref{THM-1} can be performed $n$ times, i.e., the Eckmann-Hilton construction gives rise to a complete full binary tree of depth $n$ consisting of exact couples.
	\end{thm}
	
	Here we specify these theorems for integral categories. Namely, the~following assertions hold: 
	
	\begin{thm}\label{COR-1}
		If, in Theorem \ref{THM-1}, the category $\mathcal{A}$ is assumed to be integral, then the conclusion remains valid for arbitrary morphisms $\gamma$ and $\beta$, not necessarily strict.
	\end{thm}
	
	\begin{thm}\label{COR-2}
		If, in Corollary \ref{COR-1}, the category $\mathcal{A}$ is assumed to be quasi-abelian, then the right and left derived couples coincide.
	\end{thm}
	
	\smallskip
	
	Let $\mathcal{A}$ be preabelian. We fix the exact couple \eqref{ex-coup} and assume that the morphism $\alpha$ is strict, i.e., it has a decomposition $\alpha=\rho\sigma$ with a kernel $\rho$ and a cokernel $\sigma$. Taking the exactness into account it follows that $\rho=\ker\beta$ and $\sigma=\cok\gamma$ and we can consider the diagram
	\begin{equation}\label{eck-hil1}
		\begin{tikzcd}
			D_1  & &D_1 \arrow[tail]{d}{\rho}\\
			D \arrow[two heads]{u}{\sigma} \arrow[]{r}{\beta} & E \arrow[]{r}{\gamma}&  D
		\end{tikzcd}
	\end{equation}
	with $D_1=\Im\alpha=\Coim\alpha=\Ker\beta=\Cok\gamma$. From \eqref{eck-hil1} we get the two diagrams
	\begin{equation}\label{eck-hil2}
		\begin{tikzcd}[column sep=1.1ex, row sep=1.1ex]
			D_1\arrow[]{rr}{\beta_1^-}  && E_1^-\arrow[]{rr}{\gamma_1^-} &&D_1 \arrow[equal]{dd}\\
			&\text{\footnotesize\rm PO}&&\text{\footnotesize\rm\phantom{PB}}&\\
			D\arrow[two heads]{uu}{\sigma}\arrow[]{rr}{\beta'}   && E_{\rho}\arrow[two heads]{uu}[swap]{\sigma'}\arrow[]{rr}{\gamma'}\arrow[tail]{dd}{\rho'} &&D_1 \arrow[tail]{dd}{\rho}\\
			&\text{\footnotesize\rm\phantom{PB}}&&\text{\footnotesize\rm PB}&\\
			D \arrow[equal]{uu} \arrow[]{rr}{\beta} && E \arrow[]{rr}{\gamma}&&  D
		\end{tikzcd}\hspace{30pt}
		\begin{tikzcd}[column sep=1.1ex, row sep=1.1ex]
			D_1\arrow[]{rr}{\beta_1^+} && E_1^+\arrow[]{rr}{\gamma_1^+}\arrow[tail]{dd}{\rho''}&& D_1 \arrow[tail]{dd}{\rho}\\
			&\text{\footnotesize\rm\phantom{PB}}&&\text{\footnotesize\rm PB}&\\
			D_1\arrow[equal]{uu}\arrow[]{rr}{\beta''} && E^{\sigma}\arrow[]{rr}{\gamma''} &&D\arrow[equal]{dd} \\
			&\text{\footnotesize\rm PO}&&\text{\footnotesize\rm\phantom{PB}}&\\
			D \arrow[two heads]{uu}{\sigma} \arrow[]{rr}{\beta} && E \arrow[]{rr}{\gamma}\arrow[two heads]{uu}[swap]{\sigma''}&&  D
		\end{tikzcd}
	\end{equation}
	by the classical construction of Eckmann, Hilton \cite{EH}. We start with the left diagram. We first form the pullback to obtain $\gamma'$ and $\rho'$. Then we apply its universal property and get $\beta'$ with $\gamma'\beta'=0$ and $\rho'\beta'=\beta$. Next, we form the pushout to obtain $\sigma'$ and $\beta_1^-$. By the universal property of the pushout we finally get $\gamma_1^-$ with $\gamma_1^-\beta_1^-=0$ and $\gamma_1^-\sigma'=\gamma'$. By \cite[Lemma~2.1(iii)]{KW}, $\rho'$ is a kernel and $\sigma'$ is a cokernel. The right diagram is constructed dually. The diagrams in \eqref{der-coup} are obtained by setting $\alpha_1=\sigma\rho$.
	
	\smallskip
	
	\begin{lem}\label{second}\cite[Lemmas 3 and 4]{KW2} Let $\mathcal{A}$ be left semiabelian and let $\alpha$, $\beta$ and $\rho$ be morphisms. If $\alpha=\cok\beta$ and $\im\beta=\ker(\rho\alpha)$ then $\rho$ is a monomorphism.
		
		Dually, let $\mathcal{A}$ be right semiabelian and let $\alpha$, $\beta$ and $\rho$ be morphisms. If $\coim\alpha=\cok(\rho\beta)$ and $\rho$ is a kernel then $\im\beta=\ker(\alpha\rho)$.
	\end{lem}
	
	Now, we prove Theorems \ref{COR-1} and \ref{COR-2}.
	
	\begin{proof}\textit{(of Theorem \ref{COR-1})} (i) We show $\im\beta_1^-=\ker\gamma_1^-$, $\im\alpha_1=\ker\beta_1^-$ and $\im\gamma_1^-=\ker\alpha_1$. By Lemma~\ref{second}, $\cok\beta=\coim\gamma$ and $\beta=\rho'\beta'$ imply that $\im\beta'=\ker(\gamma\rho')$. Hence, $\im\beta'=\ker(\rho\gamma')=\ker\gamma'$ and therefore $\coim\gamma'=\cok\im\beta'=\cok\beta'$. By \cite[Lemma 2.1(iii)]{KW}, we have $\cok\beta'=(\cok\beta_1^-)\sigma'$. In addition $\gamma'=\gamma_1^-\sigma'$ and $\sigma'$ is a cokernel. Therefore, $\coim\gamma'=(\coim\gamma_1^-)\sigma'$ by \cite[Cor.~2.3(ii) and Proposition~3.2]{KW}. We compute $(\cok\beta_1^-)\sigma'=\cok\beta'=\coim\gamma'= (\coim\gamma_1^-)\sigma'$ which gives $\cok\beta_1^-=\coim\gamma_1^-$ since $\sigma'$ is epic. This is equivalent to $\im\beta_1^-=\ker\gamma_1^-$.
		\smallskip
		\\We have $\beta=\rho'\beta'$, $\beta$ is strict and $\rho'$ is a kernel. Thus, $\rho'\beta'=\im(\rho'\beta')\bar{\beta}\coim(\rho'\beta')=\rho'(\im\beta')\bar{\beta}\coim(\rho'\beta')$ where the last equality follows from \cite[Corollary~2.3(i) and Proposition~3.1]{KW} and $\bar{\beta}$ is an isomorphism. We get $\beta'=(\im\beta')\bar{\beta}\coim(\rho'\beta')$ and thus $\beta'$ is strict by \cite[Remark~1.1.2(c)]{Schneiders1999}. Now we can use Proposition \ref{t1.1} to conclude that the morphism $\widehat\sigma$ defined by the equality $\sigma(\ker\beta')=(\ker\beta_1^-)\widehat\sigma$ is an epimorphism. We have $\ker\beta'=\ker(\rho'\beta')=\ker\beta=\rho$ and hence $\alpha_1=\sigma\rho=(\ker\beta_1^-)\widehat\sigma$. Therefore, $\im\alpha_1=\ker\beta_1^-$ by \cite[Corollary~2.3(i) and Proposition~3.1]{KW}.
		
		\smallskip
		
		By dualizing the last paragraph, we obtain $\cok\gamma_1^+=\coim\alpha_1$ and thus $\cok\gamma_1^-=\cok(\gamma_1^+\omega)=\cok\gamma_1^+$ with $\omega$ as in Theorem~\ref{THM-1}. Therefore, $\cok\gamma_1^-=\coim\alpha_1$ which is equivalent to $\im\gamma_1^-=\ker\alpha_1$.
		
		\smallskip
		
		The exactness of the second couple can be proved dually. The rest of the proof is exactly the same as that of \cite[Theorem 1.]{KW2}.
	\end{proof}
	
	\begin{proof}\textit{(of Theorem \ref{COR-2})}
		Since in a quasi-abelian category kernels are stable under pushouts and cokernels are stable under pullbacks, both $\ker \partial$ and $\coker \partial$ are, in particular, semistable. Hence, by \cite[Theorem~1(iii)]{KW2}, the associated derived couples coincide.
	\end{proof}

	\section{An Example of one-sided integrality}\label{Example}
	
	In this section, we present an elegant example of a non-semi-abelian preabelian category
	which was considered by Jeremy Rickard on Math StackExchange.${}^1$\footnote{%
		${}^1$ \url{https://math.stackexchange.com/questions/3631067/example-of-a-pre-abelian-category-but-not-a-semi-abelian-category}}
	Below we prove that it is right integral but not left integral.
	Since a~semi-abelian category is left integral if and only if it is right integral \cite[p.~12, Corollary]{Rump2001}, such a~category is necessarily not left semi-abelian.
	
	Let $\mathcal{C}$ be the category of diagrams
	$$
	U \xrightarrow{\alpha} V \xrightarrow{\beta} W
	$$
	of vector spaces over a fixed field $k$. Since the category of vector spaces over a fixed field is abelian, it follows that $\mathcal{C}$ is preabelian because kernels and cokernels are computed componentwise:
	$$
	\begin{tikzcd}
		\Ker f \arrow[d] \arrow[r, "\alpha''"] & \Ker g \arrow[d] \arrow[r, "\beta''"] & \Ker h \arrow[d] \\
		U \arrow[d, "f"] \arrow[r, "\alpha"]   & V \arrow[d, "g"] \arrow[r, "\beta"]   & W \arrow[d, "h"] \\
		U' \arrow[r, "\alpha'"] \arrow[d]      & V' \arrow[r, "\beta'"] \arrow[d]      & W' \arrow[d]     \\
		\Coker f \arrow[r]                     & \Coker g \arrow[r]                    & \Coker h.
	\end{tikzcd} \eqno(1)
	$$
	A morphism $(f,g,h)$ in $\mathcal{C}$ is strict if and only if each of the morphisms $f$, $g$, and $h$ is strict in the category of vector spaces. Hence $\mathcal{C}$ is an abelian category.
	
	Let $\mathcal{D}$ be the full subcategory of $\mathcal{C}$ consisting of those diagrams for which
	$$
	V=\im \alpha+\ker \beta.
	$$
	The subcategory $\mathcal{D}$ is closed under quotients in $\mathcal{C}$, and therefore cokernels in $\mathcal{D}$ exist and coincide with those in $\mathcal{C}$. It follows that $\mathcal{D}$ also has kernels. Indeed, let
	$$
	\Ker f \xrightarrow{\alpha''} \Ker g \xrightarrow{\beta''} \Ker h
	$$
	be the kernel of a morphism $(f,g,h)$ in $\mathcal{C}$. Then the kernel in $\mathcal{D}$ is given by
	$$
	\Ker f \xrightarrow{\alpha''} \im \alpha''+\ker \beta'' \xrightarrow{\beta''} \Ker h.
	$$
	Thus, $\mathcal{D}$ is a preabelian category.
	
	Now consider the following nonzero morphism:
	$$
	\begin{tikzcd}
		k \arrow[r, equal] \arrow[d, equal] & k \arrow[r, equal] \arrow[d] & k \arrow[d] \\
		k \arrow[r]                         & 0 \arrow[r]                  & 0.
	\end{tikzcd}
	$$
	Its coimage in $\mathcal{D}$ is
	$$k \to k \to 0,$$
	whereas its image is
	$$k \to 0 \to 0.$$
	Therefore, the canonical morphism from the coimage to the image is not a monomorphism. Hence $\mathcal{D}$ is not left semi-abelian, and therefore it is left integral.
	\par Observe that a morphism $(f,g,h)$ is a monomorphism in $\mathcal{D}$ if and only if it is a monomorphism in $\mathcal{C}$.
	
	Indeed, suppose that $(f,g,h)$ is a monomorphism in $\mathcal{D}$, and consider the diagram $(1)$ where
	$$
	\Ker f \xrightarrow{\alpha''} \Ker g \xrightarrow{\beta''} \Ker h
	$$
	is the kernel of $(f,g,h)$ in $\mathcal{C}$. Then
	$$
	\Ker f \xrightarrow{\alpha''} \bigl(\im \alpha''+\ker \beta''\bigr) \xrightarrow{\beta''} \Ker h
	$$
	is the kernel of $(f,g,h)$ in $\mathcal{D}$. Since $(f,g,h)$ is a monomorphism in $\mathcal{D}$, it follows that
	$$
	\Ker f=\im \alpha''+\ker \beta''=\Ker h=0.
	$$
	Since $\Ker f=0$, we have $\im \alpha''=0$, and therefore $\ker \beta''=0$. If $\Ker g\neq 0$, then
	$$
	\ker\bigl(\Ker g \xrightarrow{\beta''} 0\bigr)=\Ker g\neq 0,
	$$
	which is impossible. Hence $\Ker g=0$. Therefore $(f,g,h)$ is a monomorphism in $\mathcal{C}$. 
	\par We next show that $\mathcal{D}$ is closed under direct sums.\\
	Let
	$$
	X=(U_1 \xrightarrow{\alpha_1} V_1 \xrightarrow{\beta_1} W_1),
	\qquad
	Y=(U_2 \xrightarrow{\alpha_2} V_2 \xrightarrow{\beta_2} W_2)
	$$
	be objects of $\mathcal{D}$. Then
	$$
	X\oplus Y=
	\bigl(
	U_1\oplus U_2 \xrightarrow{\alpha_1\oplus \alpha_2}
	V_1\oplus V_2 \xrightarrow{\beta_1\oplus \beta_2}
	W_1\oplus W_2
	\bigr).
	$$
	Moreover,
	$$
	\im(\alpha_1\oplus \alpha_2)+\ker(\beta_1\oplus \beta_2)
	=
	(\im \alpha_1\oplus \im \alpha_2)+(\ker \beta_1\oplus \ker \beta_2),
	$$
	and therefore
	$$
	\im(\alpha_1\oplus \alpha_2)+\ker(\beta_1\oplus \beta_2)
	=
	(\im \alpha_1+\ker \beta_1)\oplus (\im \alpha_2+\ker \beta_2).
	$$
	Since $X,Y\in\mathcal D$, we have
	$$
	\im \alpha_1+\ker \beta_1=V_1,
	\qquad
	\im \alpha_2+\ker \beta_2=V_2.
	$$
	Hence
	$$
	\im(\alpha_1\oplus \alpha_2)+\ker(\beta_1\oplus \beta_2)=V_1\oplus V_2,
	$$
	so $X\oplus Y\in\mathcal D$. Thus $\mathcal D$ is closed under direct sums.
	
	Consequently, pushouts of morphisms in $\mathcal D$ may be computed in $\mathcal C$. Since monomorphisms are preserved and reflected, it follows that the pushout of any monomorphism is again a monomorphism. So $\mathcal{D}$ is not left integral but right integral.
	
	We now show that $\mathcal{D}$ is right quasi-abelian.
	
	Let $i\colon U\to X$ be a kernel in $\mathcal{D}$, and consider its pushout along an arbitrary morphism $g\colon U\to Z:$
	$$
	\begin{tikzcd}
		U \arrow[r, "i"] \arrow[d, "g"] \arrow[r, "\mathrm{PO}", phantom, bend right=60] & X \arrow[d] \\
		Z \arrow[r, "j"] & Y.
	\end{tikzcd}
	$$
	Since $\mathcal{C}$ is abelian, it follows that $j$ is a kernel in $\mathcal{C}$. Hence $\im_{\mathcal{C}} j = j,$ and
	$$
	\Coker_{\mathcal{C}} j=\Coker_{\mathcal{D}} j,
	$$
	because $\mathcal{D}$ is closed under quotients.
	
	Consider the canonical decomposition of $j$ in $\mathcal{C}$:
	$$
	\begin{tikzcd}
		Z \arrow[r, "j"] \arrow[rd, "\cong"'] & Y \arrow[r] & \Coker_{\mathcal{C}}j \\
		& \Im_{\mathcal{C}}j \arrow[u] &
	\end{tikzcd}
	$$
	Since both $\Im_{\mathcal{C}} j$ and $Y$ belong to $\mathcal{D}$, and $\mathcal{D}$ is a full subcategory of $\mathcal{C}$, we conclude that
	$$
	\im_{\mathcal{D}} j=\im_{\mathcal{C}} j=j.
	$$
	Therefore, $j$ is a kernel in $\mathcal{D}$. This proves that $\mathcal{D}$ is right quasi-abelian.
	
	Using this example, we obtain an example of a category that is right integral, but neither right quasi-abelian nor left semi-abelian. Indeed, it suffices to consider the product category $\mathcal{D}\times \mathbf{BOR}$ since, for product categories, such properties are inherited componentwise \cite[Theorem 3.8]{HSW2021}.

	\section*{Acknowledgements}
	The work of Ya.~Kopylov was carried out in the framework of the State Task to the
	Sobolev Institute of Mathematics (Project FWNF--2026--0026). 
	
	We are very grateful to
	Professor Sven-Ake Wegner for carefully reading the manuscript and
	pointing out several remarks.
	
	%The authors would like to thank the referee for his careful work.

\end{document}